\documentclass[12pt]{elsarticle}
\usepackage{epsfig,amsfonts}
\usepackage{amsmath}
\usepackage{epstopdf}
\usepackage{color}
\newtheorem{theorem}{Theorem}
\newtheorem{remark}[theorem]{Remark}
\newtheorem{proposition}[theorem]{Proposition}
\newenvironment{proof}[1][Proof]{\noindent\textbf{#1.} }{\ \rule{0.5em}{0.5em}}

\begin{document}
\title{Area preserving maps and volume preserving maps between a class of polyhedrons and a sphere}
\author{Adrian Holho\c s}\ead{Adrian.Holhos@math.utcluj.ro}\address{Technical University of Cluj-Napoca, Department of Mathematics, str. Memorandumului 28, RO-400114 Cluj-Napoca, Romania}
\author{Daniela Ro\c sca}\ead{Daniela.Rosca@math.utcluj.ro}\address{Technical University of Cluj-Napoca, Department of Mathematics, str. Memorandumului 28, RO-400114 Cluj-Napoca, Romania}                                           % Activate to display a given date or no date

\begin{abstract}For a class of polyhedrons denoted $\mathbb K_n(r,\varepsilon)$, we construct a bijective continuous area preserving map from $\mathbb K_n(r,\varepsilon)$ to the sphere $\mathbb S^{2}(r)$, together with its inverse. Then we investigate for which polyhedrons $\mathbb K_n(r',\varepsilon)$ the area preserving map can be used for constructing a bijective continuous volume preserving map from $\overline {\mathbb K}_n(r',\varepsilon)$ to the ball  $\overline {\mathbb S^{2}}(r)$. These maps can be further used in constructing uniform and refinable grids on the sphere and on the ball, starting from uniform and refinable grids of the polyhedrons $\mathbb K_n(r,\varepsilon)$ and $\overline {\mathbb K}_{n}(r',\varepsilon)$, respectively. In particular, we show that HEALPix grids can be obtained by mappings polyhedrons $\mathbb K_n(r,\varepsilon)$ onto the sphere.
\end{abstract}

\maketitle

\section{Introduction}

A uniform grid on a two(three)-dimensional domain $D$ is a grid all of whose cells have the same area(volume). This is required in statistical applications and in construction of wavelet bases of the space $L^2(D)$, when one wishes to use the standard inner product and 2-norm, instead of a weighted norm dependent on the grid \cite{RoRM}. A refinement process is needed for multiresolution analysis or for multigrid methods, when a grid is not fine enough to solve a problem accurately. A refinement of a grid is called uniform when each cell is divided into a given number of smaller cells having the same measure. To be efficient in practice, a refinement procedure should also be a simple one.  In many applications, especially  in geosciences, one requires simple, uniform and refinable (hierarchical) grids on the 2D-sphere or on the 3D-ball (i.e. the solid sphere). One simple method to construct grids on the 2D-sphere is to transfer existing planar grids, while for the 3D-ball a simple way would be to transfer polyhedral grids.

Partitions of the 2D-sphere ${\mathbb S^{2}}$  into regions of small diameter and equal area has already been constructed by Alexander \cite{A72}. In \cite{Leo} Leopardi derives a recursive  zonal  equal area partitioning algorithm  for the unit sphere $\mathbb S^{d}$ embedded in $\mathbb R^{d+1}$. The partition for the particular case $d=2$ consists of polar cups and rectilinear regions  that are arranged in zonal collars. Besides the fact that that the regions have different shapes, his partition is not suitable  for applications  where one must avoid  that vertices of spherical  rectangles  lie on edges of neighbor rectangles.

In astronomy, the most used construction of equal area partitions of $\mathbb S^{2}$, is the HEALPix  grid \cite{heal}, providing a hierarchical equal area iso-latitude pixelization. Other constructions are the truncated icosahedron-method of Snyder \cite{Sny}, the small circle subdivision method introduced in \cite{SKS}, the icosahedron-based method by Tegmark \cite{T96}, see also \cite{T06}. Section 1 in \cite{RP} contains a larger list of uniform spherical grids, together with their properties. A complete description of all known spherical projections from a sphere or parts of a sphere to the plane, used in cartography, is realized in \cite{Gra,Syd}. We should mention that most of the existing constructions of spherical \textit{hierarchical} (i.e. \textit{refinable}) grids do not provide an equal area partition. However, in \cite{Ro,RP,RP2} we have already constructed some area preserving maps onto the sphere, using some Lambert azimuthal equal area projections. By transporting onto the sphere uniform and refinable planar grids, we could obtain uniform and refinable spherical grids. Also, in \cite{HR} we have constructed an octahedral equal area partition of the sphere without make use of Lambert projection, and we have studied some properties of the corresponding configurations of points.
Regarding the equal volume partition of the 3D-ball, to our knowledge there was no uniform and refinable grid. This was the motivation for the work in \cite{RMG}, where we have constructed a volume preserving map from the 3D-cube to the 3D-ball. This allowed us not only the construction of uniform and refinable grids on the 3D-ball, but also a uniform sampling in the space of 3D rotations with, with applications in texture analysis.

\medskip
In this paper we consider a class of convex polyhedrons $\mathbb K_{n}(r,\varepsilon)$ and first we construct an area preserving map from $\mathbb K_{n}(r,\varepsilon)$ to the 2D-sphere of radius $r$. Then, using this map, we construct a volume preserving projection from the interior of the polyhedrons to the 3D-ball of radius $r$. The spherical grids in \cite{heal} and \cite{HR} can be obtained by mapping grids on  $\mathbb K_{n}(r,\varepsilon)$ to the 2D-sphere.

The paper is structured as follows. In Section \ref{sec:prel} we describe the class of polyhedrons $\mathbb K_n(r,\varepsilon)$ which will be projected onto the sphere and we give some formulas which will be useful in the next sections. In Section \ref{sec:conTn} we construct a continuous area preserving map from $\mathbb K_n(r,\varepsilon)$ to the sphere, together with its inverse. In Section \ref{sec:heal} we show how our maps can be useful for the construction of uniform and refinable grids on the sphere. In particular, we show that the HEALPix grids \cite{heal} can be obtained as continuous images of polyhedrons $\mathbb K_n(r,\varepsilon)$ . Finally, in Section \ref{sec:volpres}, we construct a volume preserving map from the interior of the polyhedrons $\mathbb K_n(r,\varepsilon)$ to the 3D-ball of radius $r$, using the new constructed area preserving map.

\section{Preliminaries}\label{sec:prel}
Let $r>0$ and consider the sphere $\mathbb S^2(r)=\mathbb S^2$ of radius $r$ centered at the origin $O$, of parametric equations
\begin{eqnarray}
&&x=r\cos \theta \sin \varphi,\nonumber\\
&&y=r\sin \theta \sin \varphi,\label{parsf}\\
&&z=r\cos \varphi,\nonumber
\end{eqnarray}
where $\varphi \in [0,\pi]$ is the colatitude and $\theta \in [0,2\pi)$ is the longitude. A simple calculation shows that the area element of the sphere is
\begin{equation}\label{ae}
dS=r^2\sin \varphi\, d\theta\,d\varphi.
\end{equation}

Let $\varepsilon\in[0,1)$ be fixed. We intersect the sphere $\mathbb S^2$ with the planes $z=\pm \varepsilon r$, and we denote with $\mathbb S_{r\varepsilon}^{+}=\mathbb S^+$ the spherical cap situated above the plane $z=\varepsilon r$, with $\mathbb S_{\varepsilon r}^{-}= \mathbb S^-$ the spherical cap situated below the plane $z=-\varepsilon r$, and with $\mathbb E_{\varepsilon r}=\mathbb E$ the equatorial belt between the planes $z=\pm \varepsilon r$. Their areas are
$$
\mathcal A(\mathbb S^{+})=\mathcal A(\mathbb S^{-})=2\pi (1-\varepsilon)r^2,\quad \mathcal A(\mathbb E)=4\pi \varepsilon r^2,
$$
and since $\mathbb S^{+} \cup \mathbb S^{-} \cup \mathbb E =\mathbb S^2 $ and these portions are pairwise disjoint, one has
$$
\mathcal A(\mathbb S^{+})+\mathcal A(\mathbb S^{-})+\mathcal A(\mathbb E)=\mathcal A(\mathbb S^2)=4\pi r^2.
$$
\begin{figure}\centering
\includegraphics*[width=0.4\textwidth]{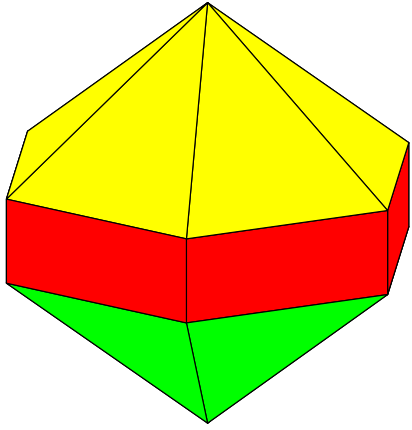}
\caption{The polyhedron $\mathbb K_n(r,\varepsilon)$ which will be projected on the sphere. Here $n=6$.}
\label{fig:idee}
\end{figure}
For a fixed integer $n$, $n\geq 3$, let $\mathbb K_n = \mathbb K_n(r,\varepsilon)$ be a polyhedron formed by a regular prism $\mathbb B_n=\mathbb B_n(r,\varepsilon)$ of height $2\varepsilon r$ and two congruent pyramids $\mathbb P_n^+(r,\varepsilon)=\mathbb P_n^+$ and $\mathbb P_n^-(r,\varepsilon)=\mathbb P_n^-$ of lateral area $2\pi(1-\varepsilon)r^2$, such that the two bases of the prism $\mathbb B_n$ coincide with the bases of the two pyramids (see Figure \ref{fig:idee}). Let $\mathcal P_n$ denotes the regular polygon which is the base of $\mathbb P_n^+$ and the upper base of $\mathbb B_n$. We place $\mathbb K_n$ such that $\mathbb B_n$ is centered at $O$ and symmetric with respect to the plane $OXY$ and such that one of its vertical edge is situated in the plane $OXZ$. Further, we divide the space $\mathbb R^3$ onto $2n$ zones $I_i^\pm$, $i=0,\ldots,n-1$, as follows:
\begin{eqnarray*}
&I_0^+=\left \{(x,y,z)\in \mathbb R^3, z\geq 0, \ 0\leq y\leq {x\tan} \frac {2\pi}n \right\},\\
&I_0^-=\left \{(x,y,z)\in \mathbb R^3, z\leq 0, \ 0\leq y\leq {x\tan} \frac {2\pi}n \right\},
\end{eqnarray*}
and further, the other zones $I_i^\pm $ are obtained by rotating $I_0^\pm$ with the angle $\alpha_i= \frac {2i\pi}n$ as
\begin{eqnarray*}
&I_i^\pm =\{ \mathcal R_i \cdot (x,y,z)^T, \ (x,y,z)\in I_0^\pm \},%\\
%&=\textcolor{red}{\{(x,y,z)\in \mathbb R^3, \pm z\geq 0, y\cos\frac{2k\pi}{n}\geq x \sin\frac{2k\pi}{n}, y\cos \frac{2(k+1)\pi}{n}\leq x \sin\frac{2(k+1)\pi}{n} \}},
\end{eqnarray*}
%\textcolor{blue}{Eu m-as feri de impartire la cos, din cauza semnului. Mai bine lasat cu tangenta}
where $\mathcal R_i$ is the $3D$ rotation matrix around $OZ$,
$$
\mathcal R_i=\left(
               \begin{array}{ccc}
                 \cos \alpha_i & -\sin \alpha_i & 0 \\
                 \sin \alpha_i & \cos \alpha_i & 0 \\
                 0 & 0 & 1 \\
               \end{array}
             \right).
$$
Thus, each face of the pyramids $\mathbb P_n^\pm$ will be situated in one of the domains $I_i^\pm$ (see Figure \ref{fig:sectorI0}).

\begin{figure}\centering
\includegraphics*[width=0.5\textwidth]{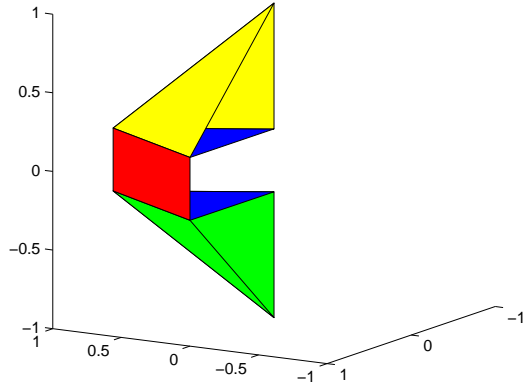}
\caption{The portions $I_0^+$ and $I_0^-$. }\label{fig:sectorI0}
\end{figure}

In Section \ref{sec:conTn} we will construct a map $\mathcal T_{n,r}=\mathcal T_n: \mathbb S^2\to\mathbb K_n $ which preserves areas, in the sense that
\begin{equation}
\mathcal A(D)=\mathcal A (\mathcal T_n(D)), \ \mbox{ for all }D\subseteq \mathbb S^2,
\end{equation}
where $\mathcal A(D)$ denotes the area of a domain $D$.
For an arbitrary point $(x,y,z)\in \mathbb S^2$ we denote
\begin{equation}\label{XYZ}
(X,Y,Z)=\mathcal T_n(x,y,z)\in \mathbb K_n.
\end{equation}

The restrictions of $\mathcal T_n$ to $\mathbb S^+$, $\mathbb S^-$, $\mathbb E$ will be denoted by $\mathcal T_n^+$, $\mathcal T_n^-$ and $\mathcal T_n^e$, respectively. More precisely, we will deduce the formulas for the area preserving maps $\mathcal T_n^+:\mathbb S^+ \to \mathbb P_n^+$, $\mathcal T_n^-:\mathbb S^- \to \mathbb P_n^-$ and $\mathcal T_n^e:\mathbb E^+ \to \mathbb B_n.$

We introduce the following notations:
\begin{eqnarray*}
%2E&=&\mbox{the height of the prism $\mathbb B_n$},\\
R_n&=&\mbox{the radius of the circle circumscribed to }\mathcal P_n,\\
r_n&=&\mbox{the radius of the circle inscribed in }\mathcal P_n,\\
b_n&=& \mbox{the altitude of the pyramid } \mathbb P_n^+,\\
\ell_n&=&\mbox{the edge of the polygon }\mathcal P_n,\\
%L_n&=&\mbox{the edge of the lateral face of the pyramid }\mathbb P_n^+\\
%\alpha_n&=&\mbox{the angle between }L_n \mbox{ and the plane of }\mathcal P_n\\
%\gamma_n&=&\mbox{the angle between }\ell_n \mbox{ and }L_n\\
a_n &=&\mbox{the slant height of the pyramid }\mathbb P_n^+.
\end{eqnarray*}
We impose that the lateral area of $\mathbb B_n$ equals the area of the equatorial belt $\mathbb E$, therefore
\begin{equation}\label{ell}
\ell_n=\frac{2\pi r}n.
\end{equation}
On the other hand, since $\ell_n=2R_n \sin \frac \pi n$, we deduce that
\begin{equation}\label{Rn}
R_n=\frac{\pi r}{n\sin \frac \pi n}.
\end{equation}
We also have
\begin{flalign}\label{rn}
r_{n}&=R_{n} \cos \frac{\pi}{n}.
\end{flalign}
%\begin{equation}\label{InRn}
%I_n=R_n \tan\alpha_n,
%\end{equation}
%$$
 %r_n=R_n \cos \frac \pi n,\quad a_n^2=r_n^2+I_n^2=R_n^2(\cos^2 \frac {\pi} n+\tan^2 \alpha_n).
%$$
%\begin{equation}\label{singamma}
%\sin \gamma_n =\frac{a_n}{L_n}=\frac{a_n}{\sqrt{I_n^2+R_n^2}}=\left(\frac{\cos^2 \frac{\pi}{n}+\tan^2 \alpha_n}{1+\tan^2 \alpha_n} \right)^{1/2}=\cos \alpha_n \left(\cos^2 \frac \pi n+\tan^2 \alpha_n \right)^{1/2}.
%\end{equation}
The area of a face of the pyramid $\mathbb P_n^+$ is
\begin{equation}\label{psin}
\mathcal A_n=\frac{a_n\cdot \ell_n}{2},%=R_n^2 \psi_n,\ \mbox{ with }\psi_n=\sin \frac \pi n \left(\cos^2 \frac{\pi}{n}+\tan^2 \alpha_n\right)^{1/2},
\end{equation}
and by imposing that $\mathcal A(\mathbb P_n^+)=\mathcal A (\mathbb S^+) $, we find that $a_{n}=2(1-\varepsilon)r$ and
\begin{equation}\label{An}
\mathcal A_n=\frac{2\pi(1-\varepsilon)r^2}n=\ell_{n}(1-\varepsilon)r.
\end{equation}
Using the equality $a_{n}^2=r_{n}^2+b_{n}^2$, we obtain
\begin{flalign}\label{In}
%b_{n}=\frac \pi{n\sin \frac \pi n}\left( (1-\varepsilon)^2\frac{4n^2}{\pi^2}\sin^2\frac \pi n-\cos^2 \frac \pi n\right)^{1/2}.
b_n=r \left( 4(1-\varepsilon)^2-\frac{\pi^2}{n^2}\cot^2 \frac \pi n\right)^{1/2}.
\end{flalign}
\section{Construction of the area preserving map $\mathcal T_n:\mathbb S^2 \to \mathbb K_n$ and its inverse}\label{sec:conTn}
\subsection{Construction of the map $\mathcal T_n^+$}
For $h\in(1-\varepsilon,1)$, we denote by $\mathbb S_{rh}$ the spherical cap situated above the plane of equation $z=rh$. A simple calculation shows that
\begin{equation}\label{ariash}
\mathcal A(\mathbb S_{rh})=2\pi(1-h)r^2.
\end{equation}
Now, we calculate $H>0$ such that the portion $\mathbb P_n^+(H)$ of the pyramid $\mathbb P_n^+$ situated above the plane $z=H$ has the same area $\mathcal A(\mathbb S_{rh})$. For the small pyramid $\mathbb P_n^+(H)$, let $\widetilde{R}_n$ be the radius of the circumscribed circle of its base. Then, one has
$$
\frac {\widetilde{R}_n}{R_n}=\frac{b_n-H+\varepsilon r}{b_n},
$$
and the area $\mathcal A(H)$ of a face of $\mathbb P_n^+(H)$ is therefore
$$
\mathcal A_n(H)=\mathcal A_n\cdot\left( 1-\frac{H-\varepsilon r}{b_n}\right)^2= \frac{2\pi(1-\varepsilon)r^2}n \left( 1-\frac{H-\varepsilon r}{b_n}\right)^2.
$$
Imposing now $\mathcal A(\mathbb S_{rh})=n \mathcal A(H), $ we obtain that
$$
2\pi (1-h)r^2=2\pi (1-\varepsilon)r^2\left(1-\frac{H-\varepsilon r}{b_n} \right)^2,
$$
whence after some calculations we obtain
\begin{equation}\nonumber
H=\varepsilon r+b_n\left(1-\sqrt{\frac{1-h}{1-\varepsilon}} \right).
\end{equation}
In conclusion, for $z\geq \varepsilon r$ we can define $Z$ from \eqref{XYZ} as
\begin{equation}\label{defZ}
Z=\varepsilon r+b_n\left(1-\sqrt{\frac{1-z/r}{1-\varepsilon}} \right),
\end{equation}
meaning that the parallel circles of the sphere map onto polygons obtained as intersections of $\mathbb P_n^+$ with planes parallel to $XOY$.
%In order to find the expression of $I_n$ with respect to $n$ and $E$,
%from \eqref{psin} and \eqref{An} we deduce that
%$$
%\frac{4\pi^2(1-E)^2}{n^2}=R_n^2 \frac {\pi^2}{n^2}\left(cos^2 \frac \pi n + \tan^2 \alpha_n \right),
%$$
%and then, replacing $R_n$ from \eqref{Rn} we find
%$$
%\tan^2 \alpha_n=\frac{4(1-E)^2}{R_n^2}-\cos^2 \frac \pi n,
%$$
%and finally, from \eqref{InRn} we have
%\begin{equation}\label{In}
%I_n=\frac \pi{n\sin \frac \pi n}\left( (1-E)^2\frac{4n^2}{\pi^2}\sin^2\frac \pi n-\cos^2 \frac \pi n\right)^{1/2}.
%\end{equation}
%\textcolor{blue}{sper ca-s bune calculele pt $I_n$}

For the formulas for $X$ and $Y$ we proceed as follows.
We focus on the portion $\mathcal F_0^+$ of $\mathbb S^2$, situated in $I_0^+$, and the face of the pyramid $\mathbb P_n^+$ situated in $I_0^+$ will be denoted with $F_0^+.$ We denote $A(\sqrt{1-\varepsilon^2}r,0,\varepsilon r),$ $C(0,0,r)$ and we consider the vertical plane of equation $y=x\tan \alpha,$ with $\alpha \in (0,2\pi/n)$ (see Figure \ref{fig:expl}, left). We denote by $\widetilde{MC}$ its intersection with $\mathcal F_0^+$. More precisely, $M(\sqrt{1-\varepsilon^2}\, r \cos \alpha,\sqrt{1-\varepsilon^2}\,r \sin \alpha,\varepsilon r).$ The area of the spherical domain $AMC$ equals $\alpha(1-\varepsilon)r^2.$ Now we intersect the face $F_0^+$ of the pyramid with the vertical plane of equation $y=x\tan \beta$ and we denote by $M'(X',Y')$ its intersection with the edge $A'B'$, where $A'(R_n,0,\varepsilon r)$ and $B'(R_n\cos 2\pi/n, R_n \sin 2\pi/n, \varepsilon r)$ (see Figure \ref{fig:expl}, right).
\begin{figure}
\includegraphics*[width=0.45\textwidth]{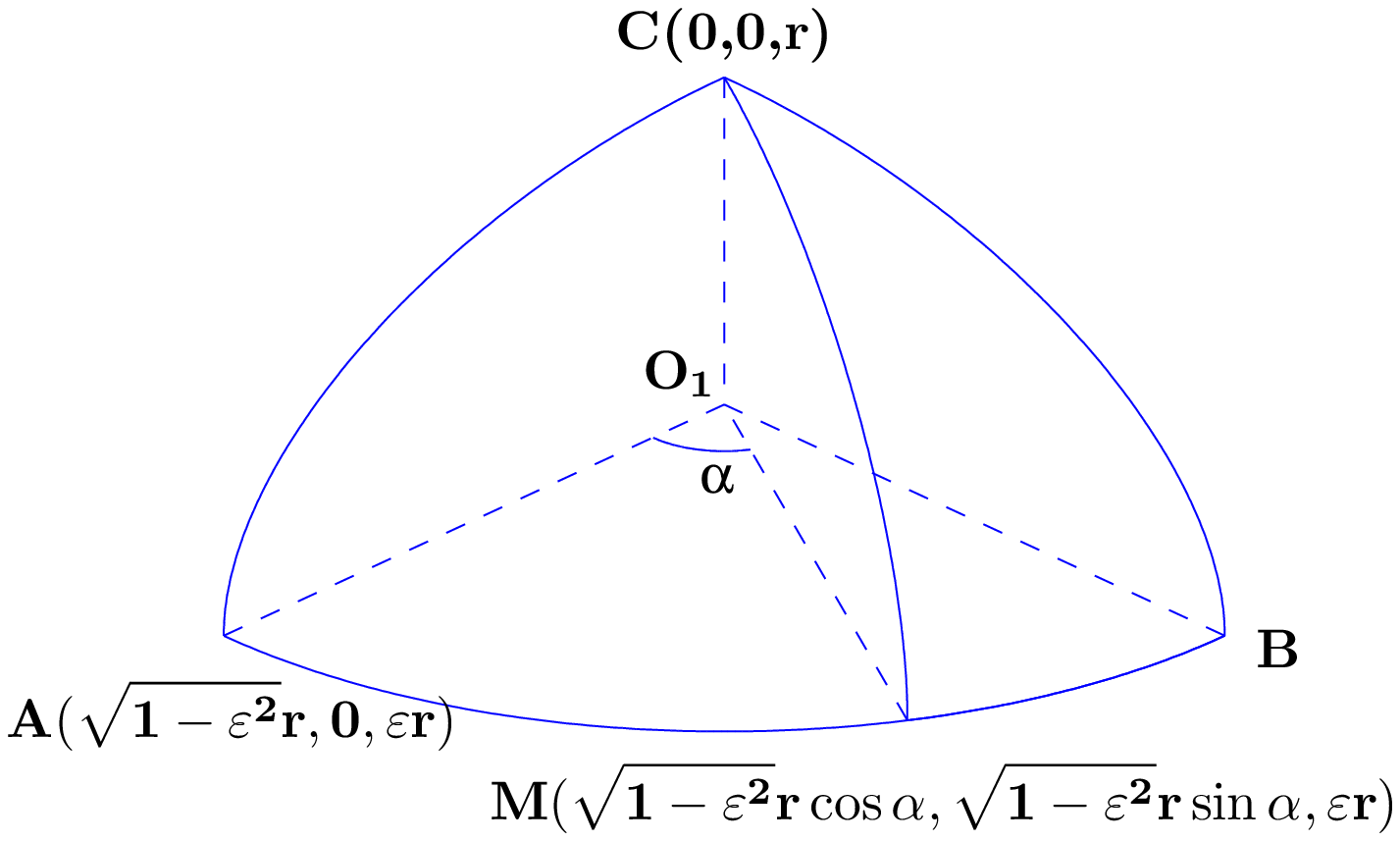}
\includegraphics*[width=0.45\textwidth]{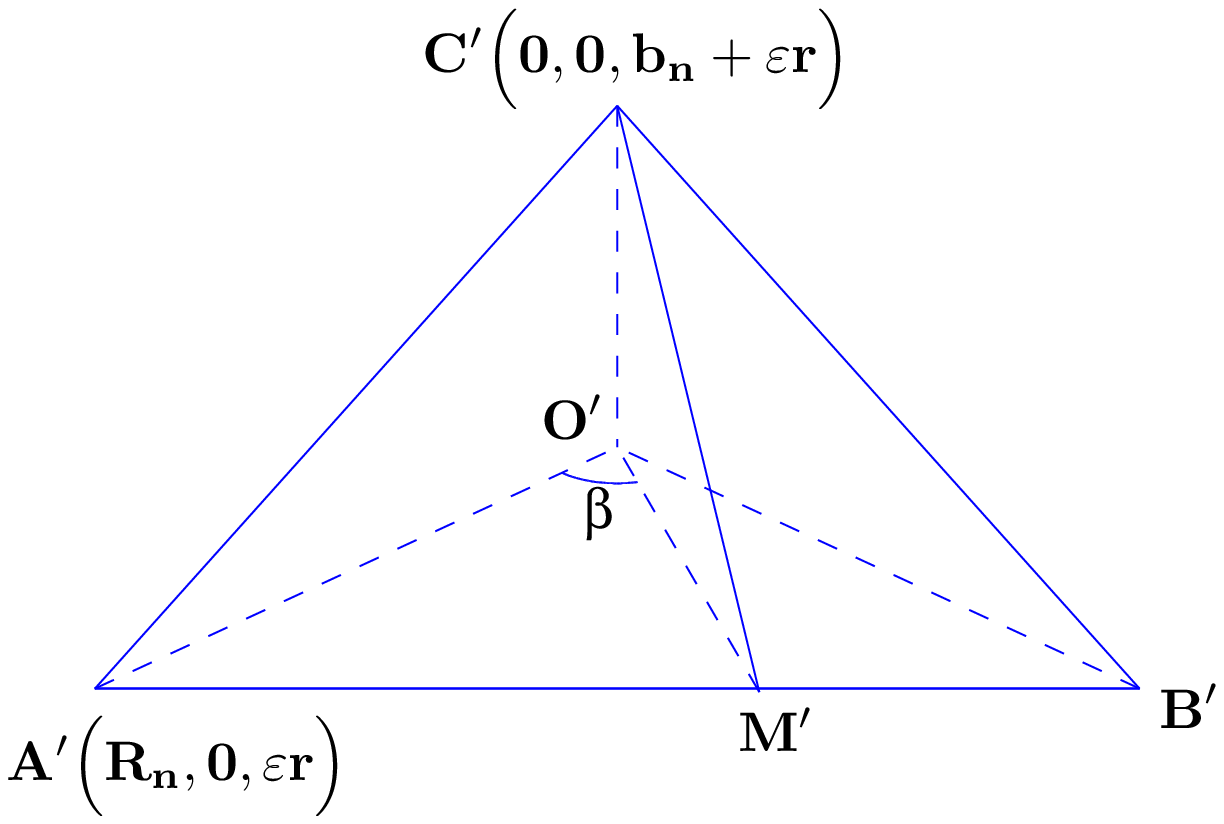}
\caption{The portion $AMC$ of the spherical cap $\mathbb S^+$ and its image triangle $A'M'C'=\mathcal T_n(AMC)$ on the pyramid $\mathbb P_n^+$.}\label{fig:expl}
\end{figure}

For finding the coordinates $X',Y'$ we use the fact that $M'$ is the intersection of the line $A'B'$ with the plane of equation $Y'=X'\tan \beta$, so solving the system
$$
\frac{X'-R_n}{R_n \left(\cos \frac {2\pi}n-1 \right)}=\frac{Y'-0}{R_n \sin \frac {2\pi}n},\qquad Y'=X'\tan \beta,
$$
we obtain
\begin{equation}
X'=\frac {R_n}{1+\tan \beta \tan \frac \pi n},\qquad Y'=\frac {R_n \tan \beta}{1+\tan \beta \tan \frac \pi n}.
\end{equation}
Further,
$$
A'M'=\frac{Y'}{\cos \frac \pi n},%\qquad A'C'=\frac {R_n}{\cos \alpha_n},
$$
and therefore, for the area of the planar triangle $A'C'M'$ we find%, using also \eqref{singamma},
$$
\mathcal A(A'C'M')=\frac{A'M'\cdot a_{n}}{2}=\frac{2\pi(1-\varepsilon)r^2\tan \beta}{n\sin\frac{2\pi}{n}\big(1+\tan \beta \tan \frac \pi n \big)}.
$$
If we impose that this area equals the area of the spherical domain $AMC$, which is $\alpha(1-\varepsilon)r^2$, we finally find that
\begin{equation}\label{beta}
\tan \beta =\frac {\alpha n \sin \frac {2\pi}n}{2\big(\pi-\alpha n\sin^2 \frac \pi n \big)}.
\end{equation}
This is in fact the relation between $\alpha$ and $\beta$, such that the slice $\widetilde{ACM}$ of $\mathcal F_0^+$ has the same area with the triangular slice $A'C'M'$ of the face $F_0^+.$

Consider now $V=(m,q,p)\in \mathcal F_0^+$ and denote $N=(u,s,t):=\mathcal T_n(V).$ The point $V$ is the intersection of $\mathcal F_0^+$ with the planes
$$
(P_1):z=p\ \mbox{ and }\ (P_2):y=\frac qm x,\ \mbox{ where }\frac qm=\tan \alpha.
$$
The point $N$ will be the intersection of $\mathcal T_n(P_1)$, $\mathcal T_n(P_2)$ and the face $F_0^+$, therefore
\begin{eqnarray*}
\mathcal T_n(P_1)&:& t=\varepsilon r+b_n\left(1-\sqrt{\frac{1-p/r}{1-\varepsilon}} \right),\\
\mathcal T_n(P_2)&:& s= \frac{n \sin \frac{2\pi}{n}\cdot \arctan \frac qm}{2\left(\pi- n\arctan \frac qm \sin^2 \frac \pi n\right)}\cdot u,\\F_0^+&:& u\cdot b_n+s\cdot b_n \tan \frac {\pi}n+t\cdot R_n =(b_n+\varepsilon r)R_n.
\end{eqnarray*}
Solving this system we find
\begin{eqnarray*}
u&=&R_{n}\sqrt {\frac{1-p/r}{1-\varepsilon}}\left(1-\frac n\pi \sin^2\frac \pi n \arctan \frac qm \right),\\
s&=&R_{n}\sqrt {\frac{1-p/r}{1-\varepsilon}}\,\frac{n\sin \frac {2\pi}{n}\arctan \frac qm}{2\pi},\\
t&=&\varepsilon r+b_n\left( 1-\sqrt \frac{1-p/r}{1-\varepsilon} \right).
\end{eqnarray*}
In conclusion, using \eqref{Rn} and with the notations in \eqref{XYZ}, for $(x,y,z)\in \mathcal F_0^+$, we define the restriction of $\mathcal T_n^+$ to $\mathcal F_0^+$ as
\begin{eqnarray}
X&=&\sqrt {\frac{r(r-z)}{1-\varepsilon}}\left(\frac{\pi}{n\sin \frac \pi n}- \sin\frac \pi n \arctan \frac yx \right),\label{fX}\\
Y&=&\sqrt {\frac{r(r-z)}{1-\varepsilon}}\,\cos \frac {\pi}{n}\arctan \frac yx,\label{fY}\\
Z&=&\varepsilon r+b_n\left( 1-\sqrt \frac{1-z/r}{1-\varepsilon} \right),\label{fZ}
\end{eqnarray}
with $b_n$ given in \eqref{In}.

In general, for $i=0,\ldots,n-1$, let $\mathcal F_i^\pm$ be the portion of $\mathbb S^2$ situated in $I_i^\pm$. The the restriction of $\mathcal T_n^+$ to $\mathcal F_i^+$ will be calculated as
\begin{equation}\label{frot}
(X,Y,Z)^T=\mathcal R_i\cdot \mathcal T_n^+\left( \mathcal R_i^{T}\cdot (x,y,z)^T\right),\quad \mbox{for } (x,y,z)\in \mathcal F_i^+,\ \mbox{ i.e.}
\end{equation}
\begin{eqnarray}
X&=&\sqrt {\frac{r(r-z)}{1-\varepsilon}}\left(\frac{\pi \cos \alpha_i}{n\sin \frac{\pi}{n}}- \sin(\alpha_i+ \frac{\pi}{n}) \cdot\arctan \frac{-x\sin \alpha_i+y\cos \alpha_i}{x\cos\alpha_i+y\sin\alpha_i} \right),\label{fiplusX}\\
Y&=&\sqrt {\frac{r(r-z)}{1-\varepsilon}}\left(\frac{\pi \sin \alpha_i}{n\sin \frac{ \pi} n}+ \cos(\alpha_i+ \frac{\pi}{n}) \cdot\arctan \frac{-x\sin \alpha_i+y\cos \alpha_i}{x\cos\alpha_i+y\sin\alpha_i} \right),\label{fiplusY}\\
Z&=&\varepsilon r+b_n\left( 1-\sqrt \frac{1-z/r}{1-\varepsilon} \right).\label{fiplusZ}
\end{eqnarray}

In spherical coordinates $(\varphi,\theta),$ the map $\mathcal T_n^+$ restricted to $\mathcal F_0^+$ writes as
\begin{eqnarray*}
X&=&r\sqrt{\frac{2}{1-\varepsilon}} \left(\frac{\pi}{n \sin \frac \pi n} -\theta \sin \frac \pi n\right) \sin \frac \varphi 2,\\
Y&=&r\sqrt{\frac{2}{1-\varepsilon}}\, \theta \cos \frac \pi n \sin \frac \varphi 2,\\
Z&=&\varepsilon r+b_n\left( 1-\sqrt \frac{2}{1-\varepsilon} \sin \frac \varphi 2 \right).
\end{eqnarray*}

For the portion $\mathcal F_i^+$ the map is defined by
\begin{eqnarray*}
X&=&r\sqrt{\frac{2}{1-\varepsilon}} \left(\frac{\pi\cos \alpha_i}{n \sin \frac \pi n} -\sin \frac{\pi(2i+1)}{n}\cdot \left( \theta- \alpha_i\right)\right) \sin \frac \varphi 2,\\
Y&=&r\sqrt{\frac{2}{1-\varepsilon}} \left(\frac{\pi\sin\alpha_i}{n \sin \frac \pi n} +\cos \frac{\pi(2i+1)}{n}\cdot \left( \theta- \alpha_i\right)\right) \sin \frac \varphi 2,\\
Z&=&\varepsilon r +b_n\left( 1-\sqrt \frac{2}{1-\varepsilon} \sin \frac \varphi 2 \right).
\end{eqnarray*}

If we evaluate the coefficients of the first fundamental order of this surface
\begin{eqnarray*}
&&E'=(X'_{\varphi})^2+(Y'_{\varphi})^2+(Z'_{\varphi})^2=\frac{r^2}{2(1-\varepsilon)}\; \left[4(1-\varepsilon)^2 +\left(\theta- \frac{(2i+1)\pi}{n} \right)^2 \right]\cos ^{2}\frac{\varphi}{2},\\
&&F'=X'_{\varphi}X'_{\theta}+Y'_{\varphi}Y'_{\theta}+Z'_{\varphi}Z'_{\theta}=\frac{r^2}{2(1-\varepsilon)}\; \left( \theta - \frac{(2i+1)\pi}{n}\right) \sin\varphi,\\
&&G'=(X'_{\theta})^2+(Y'_{\theta})^2+(Z'_{\theta})^2=\frac{2r^2}{1-\varepsilon}\; \sin^2 \frac{\varphi}{2},
\end{eqnarray*}
then we find that $E'G'-(F')^2=r^4\sin^2 \varphi$, which is equal to $EG-F^2$ for the sphere (see \eqref{ae}). This property holds for all the portions $\mathcal F_i^\pm$, in conclusion, $\sqrt{EG-F^2}$ is invariant under $\mathcal T_n^+$, whence $\mathcal T_n^+$ is an area preserving map.
\begin{remark}The map $\mathcal T_n^-:\mathbb S^- \to \mathbb P_n^-$ can be obtained by symmetry. More precisely, the formulas are
\begin{eqnarray}
X&=&\sqrt {\frac{r(r+z)}{1-\varepsilon}}\left(\frac{\pi}{n\sin \frac \pi n}- \sin\frac \pi n \arctan \frac yx \right),\label{fX-}\\
Y&=&\sqrt {\frac{r(r+z)}{1-\varepsilon}}\,\cos \frac {\pi}{n}\arctan \frac yx,\label{fY-}\\
Z&=&\varepsilon r+b_n\left( 1-\sqrt \frac{1+z/r}{1-\varepsilon} \right),\label{fZ-}
\end{eqnarray}
\end{remark}
\subsection{The inverse of the map $\mathcal T_n^+$}

Let $(X,Y,Z)\in \mathbb P_n^+$. First, from \eqref{fiplusZ} we immediately deduce that
\begin{equation}\label{fz}
z=r-r(1-\varepsilon)\left(1-\frac{Z-\varepsilon r}{b_n}\right)^2.
\end{equation}

Then, from \eqref{fiplusX} and \eqref{fiplusY} we find that
$$
\frac YX=\frac{\frac{\pi \sin \alpha_i}{n\sin \frac{\pi}{n}}+\cos (\alpha_i+ \frac{\pi}{n})\cdot\arctan \frac{-x\sin \alpha_i+y\cos \alpha_i}{x\cos\alpha_i+y\sin\alpha_i}}{ \frac{\pi \cos \alpha_i}{n\sin \frac{\pi}{n}}-\sin (\alpha_i+ \frac{\pi}{n})\cdot\arctan \frac{-x\sin \alpha_i+y\cos \alpha_i}{x\cos\alpha_i+y\sin\alpha_i}},
$$
whence
$$
\arctan \frac{-x\sin \alpha_i+y\cos \alpha_i}{x\cos\alpha_i+y\sin\alpha_i}=\frac{ \pi}{n\sin \frac \pi n}\cdot \frac{-X \sin\alpha_i+Y\cos\alpha_i}{X\cos (\alpha_i+ \frac{\pi}{n})+Y\sin(\alpha_i+ \frac{\pi}{n})}=:\lambda.
$$
Therefore, $$y=x \,\frac{\tan\lambda \cos\alpha_i+\sin\alpha_i}{\cos\alpha_i-\tan\lambda \sin\alpha_i},$$ and then, taking into account that $x^2+y^2+z^2=r^2,$ after simple calculations we find
\begin{eqnarray}
x&=&\sqrt{r^2-z^2}\cdot \cos \left(\frac{ \pi}{n\sin \frac \pi n}\cdot \frac{-X \sin\alpha_i+Y\cos\alpha_i}{X\cos (\alpha_i+ \frac{\pi}{n})+Y\sin(\alpha_i+ \frac{\pi}{n})}+ \alpha_i \right),\label{fx}\\
y&=&\sqrt{r^2-z^2}\cdot \sin \left(\frac{ \pi}{n\sin \frac \pi n}\cdot \frac{-X \sin\alpha_i+Y\cos\alpha_i}{X\cos (\alpha_i+ \frac{\pi}{n})+Y\sin(\alpha_i+ \frac{\pi}{n})}+ \alpha_i \right).\label{fy}
\end{eqnarray}
where $z$ is given in \eqref{fz}.
\begin{remark}For the inverse of $\mathcal T_n^-$, the formulas for $x$ and $y$ are as in \eqref{fx} and \eqref{fy}, respectively, while for $z$ formula \eqref{fz} changes into
\begin{equation}\label{fzminus}
z=r(1-\varepsilon)\left(1-\frac{Z-\varepsilon}{b_n}\right)^2-r.
\end{equation}
\end{remark}
\begin{remark} Formulas \eqref{fx} and \eqref{fy} can also be obtained from the formulas for $i=0$, similarly to \eqref{frot}.
\end{remark}

\subsection{Construction of the maps $(\mathcal T_n^e)^{-1}$ and $\mathcal T_n^e$}

The prism $\mathbb B_n$ is regular with height $2\varepsilon r$ and edge $\ell_n$ given in \eqref{ell}.
The face $B_0$ of $\mathbb B_n$ situated in the zone $I_0=I_0^+\cup I_0^-$ has the equation
\begin{equation}\label{ecfpr}
Y \cdot \sin \frac \pi n + X\cdot\cos \frac \pi n=R_n \cos \frac \pi n,
\end{equation}
or, equivalently, $Y=(R_n-X) \cot \frac \pi n$. Then, in order to project the face $B_0$ onto the plane $OYZ$, we make a translation with $R_n$ along the axis $OX$, followed by a rotation of angle $-\frac \pi n$ around the axis $OZ$. Thus, the point $$(X,(R_n-X)\cot \frac \pi n,Z) \in \mathbb B_n$$ will be first translated to $$(X-R_n,(R_n-X)\cot \frac \pi n,Z) \in \mathbb B_n,$$
and then rotated around $OZ$ it will be mapped onto the point
 $$
P\left(0, \frac{R_n-X}{\sin \frac \pi n},Z \right)=P\left(0, \frac{Y}{\cos \frac \pi n},Z \right)
 $$
As a translation followed by rotation, this map will be area preserving. Next we use the inverse Lambert cylindrical equal area projection (see i.e. \cite{Gra}), the point $P$ being further mapped onto

\begin{equation}\label{mape}
 \left(\sqrt {r^2-Z^2}\,\cos \frac {Y}{r\cos \frac \pi n}, \sqrt {r^2-Z^2}\,\sin \frac {Y}{r\cos \frac \pi n},Z\right)\in \mathbb E \subset \mathbb S^2.
\end{equation}

 In conclusion, the map $(\mathcal T_n^e)^{-1}:\mathbb B_n \to \mathbb E$ is area preserving and maps the point $(X,Y,Z)\in \mathbb B_n$ onto the point given in \eqref{mape}. \par
In general, similarly to \eqref{frot}, the face of the prism $\mathbb B_n$ situated in  $I_i=I_i^+\cup I_i^-$ is mapped on the sphere by

%\begin{flalign}
%x&=\sqrt{1-Z^2}\; \cos\left(\frac{R_{n}-X\;\cos \alpha_i-Y\;\sin\alpha_i}{\sin\frac{\pi}{n}}+\alpha_i \right),\label{xeq}\\
%y&=\sqrt{1-Z^2}\; \sin\left(\frac{R_{n}-X\;\cos \alpha_i-Y\;\sin\alpha_i}{\sin\frac{\pi}{n}}+\alpha_i \right),\label{yeq}\\
%z&=Z.\label{zeq}
%\end{flalign}

\begin{flalign}
x&=\sqrt{r^2-Z^2}\; \cos\left(\frac{Y\;\cos \alpha_i-X\;\sin\alpha_i}{r\cos\frac{\pi}{n}}+\alpha_i \right),\label{xeq}\\
y&=\sqrt{r^2-Z^2}\; \sin\left(\frac{Y\;\cos \alpha_i-X\;\sin\alpha_i}{r\cos\frac{\pi}{n}}+\alpha_i \right),\label{yeq}\\
z&=Z.\label{zeq}
\end{flalign}

For the direct application $\mathcal T_n^e:\mathbb E \to \mathbb B_n$, the calculations give

\begin{eqnarray}
X&=&R_{n}\cos \alpha_i- r \sin (\alpha_i+\frac{\pi}{n}) \arctan \frac{-x\sin \alpha_i+y\cos \alpha_i}{x \cos \alpha_i+y \sin \alpha_i},\label{fe1}\\
Y&=&R_{n}\sin \alpha_i+r \cos (\alpha_i+\frac{\pi}{n}) \arctan \frac{-x\sin \alpha_i+y\cos \alpha_i}{x \cos \alpha_i+y \sin \alpha_i},\label{fe2}\\
Z&=&z.\label{fe3}
\end{eqnarray}
%----------------
 \subsection{The continuity of $\mathcal T_n$}
 \begin{proposition}The map $\mathcal T_n:\mathbb S^2 \to \mathbb K_n$ is continuous.
 \end{proposition}
 \begin{proof}
 %It is enough to prove the continuity of $\mathcal T_{n}$ on the joint edge of the prism $\mathbb B_{n}$ and the pyramid $\mathbb P_{n}^{+}$ which is in the zone $I_{0}^{+}$.
It is enough to restrict ourselves to the first octant $I_{0}^{+}$ and to prove the continuity of $\mathcal T_{n}^{-1}$ on $\mathbb B_{n} \cap \mathbb P_{n}^{+}\cap I_{0}^{+}$.

For $Z=r\varepsilon$, the point projected onto the border of $\mathbb E$ is
 \begin{equation}\label{f1}
 \left (r\sqrt {1-\varepsilon^2}\, \cos \frac {Y}{r\cos \frac \pi n},r\sqrt {1-\varepsilon^2}\, \sin \frac {Y}{r\cos \frac \pi n} ,\varepsilon r\right ).
 \end{equation}
On the other hand, let $(X,Y,\varepsilon r)\in \mathbb P_n^+ $. From \eqref{fx}, \eqref{fy} and \eqref{fz}, its image on the sphere, $\mathcal T_n(X,Y,\varepsilon r)$, will be
\begin{eqnarray*}
 &&x=r\sqrt {1-\varepsilon^2}\,\cos \frac {Y \pi}{n \sin \frac \pi n\cdot \left(X \cos \frac \pi n+Y \sin \frac \pi n \right)},\\
&&y=r\sqrt {1-\varepsilon^2}\, \sin \frac {Y \pi}{n \sin \frac \pi n\cdot \left(X \cos \frac \pi n+Y \sin \frac \pi n \right)},\\
&&z=\varepsilon r.
\end{eqnarray*}
But $(X,Y,\varepsilon r)$ is also a point of $\mathbb B_n$, so it satisfies \eqref{ecfpr}, and finally, replacing $R_n$ from \eqref{Rn}, the above formulas transform into
\begin{eqnarray*}
 &&x=r\sqrt {1-\varepsilon^2}\cos \frac {Y }{r\cos \frac \pi n},\\
&&y=r\sqrt {1-\varepsilon^2}\sin \frac {Y }{r\cos \frac \pi n},\\
&&z=\varepsilon r,
\end{eqnarray*}
i.e. the same as in formula \eqref{f1}.

Similarly one can prove the same equality for a point in $\mathbb P_n^+\cap \mathbb B_n$ and then in $\mathbb P_n^- \cap \mathbb B_n$. In conclusion, the map $\mathcal T_n^{-1}$ is continuous, therefore $\mathcal T_n$ is also continuous.
 \end{proof}

%-----------------
\subsection{Uniform grids on the sphere}\label{sec:heal}
The most important application of our area preserving projection maps is the construction of spherical grids by mapping grids from the polyhedrons $\mathbb K_n(r,\varepsilon)$ to the sphere $\mathbb S^{2}(r)$. The major advantage is the possibility of construction of uniform and refinable spherical grids, since uniform and refinable grids are easier obtainable on polyhedrons.

In the particular case $n=4$ and $\varepsilon=0$, the polyhedron $\mathbb K_{4}(r,0)$ is an octahedron, but not a regular one, as in \cite{HR}. However, one can easily see that, if we make uniform triangular refinements by dividing each planar triangle into four small triangles by lines parallel to the edges, one obtains exactly the grids in \cite{HR}.

Next we will show that the HEALPix grids (see \cite{heal}) can also be obtained by mapping onto the sphere grids on $\mathbb K_{n}(\varepsilon,r)$. However, one advantage of our map is that it allows us to transport onto the sphere some already constructed continuous wavelet bases on polyhedrons, by a similar technique as the one described in \cite{RMC}. Compared with the construction in \cite{RMC}, one can use now the usual scalar product in $L^{2}(\mathbb S^{2})$ instead of a weighted scalar product, since the map is area preserving. Moreover, the transportation of continuous functions on  $\mathbb K_{n}(\varepsilon,r)$ yields continuous non-distorted functions on the sphere, and this does not usually happen when functions defined on rectangles are mapped onto the sphere with the area preserving maps used in \cite{heal}. When transporting a function $f:[0,a]\times[-b,b]\to \mathbb R$ onto the sphere using the maps used in \cite{heal}, the transported function defined on $\mathbb S^{2}$ is usually distorted, and it remains continuous on $\mathbb S^{2}$ only when it satisfies some periodicity conditions (see \cite{RMC}).

\subsubsection{Obtaining the value of $\varepsilon$}\label{sec:healp}

In order to obtain the HEALPix grids on the sphere, we consider $r=1$ for simplicity, and we construct the grids on the polyhedron $\mathbb K_{n}(1,\varepsilon)$ as follows.

Let $p\in \mathbb N$.
For $p$ even, on each face of the prism $\mathbb B_{n}$ we first draw horizontal lines to divide each face into $p/2$ rectangles having the same area. Then, for each of the rectangle we draw the diagonals. Together with the triangular faces of the pyramids of $\mathbb K_{n}$, they will form a grid consisting in $n\cdot p$ cells of rhombic shape (see Figure \ref{fig:gplanppar}), and we will impose that all this cells  have the same area $A_{\mathrm{cell}}$, and this area will be the same after projecting the grids onto the sphere. Therefore we have
\begin{equation}\label{deteps}
\mathcal A(\mathbb E_{\varepsilon})=4\pi \varepsilon=n p \mathcal A_{\mathrm{cell}},\qquad  \mathcal A(\mathbb S_{\varepsilon}^{+})=2\pi(1-\varepsilon)=\frac n 2 \mathcal A_{\mathrm{cell}}.
\end{equation}
Elimining $A_{\mathrm{cell}}$ from these equalities we obtain $\varepsilon=\frac{p}{p+1}$.%, like in the case of HEALPix grid.
\begin{figure}
\centering
\includegraphics[width=0.5\columnwidth]{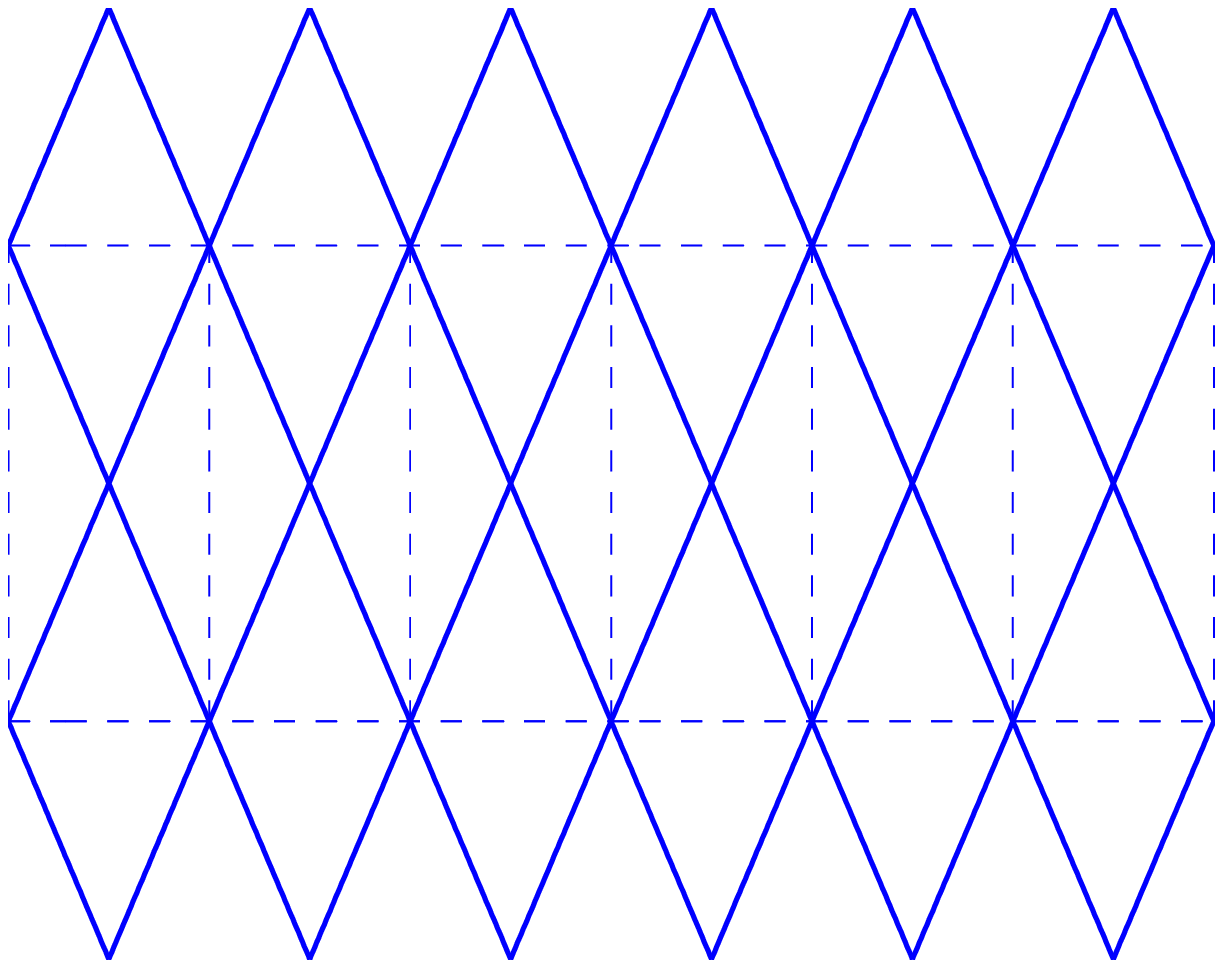}%
\includegraphics[width=0.5\columnwidth]{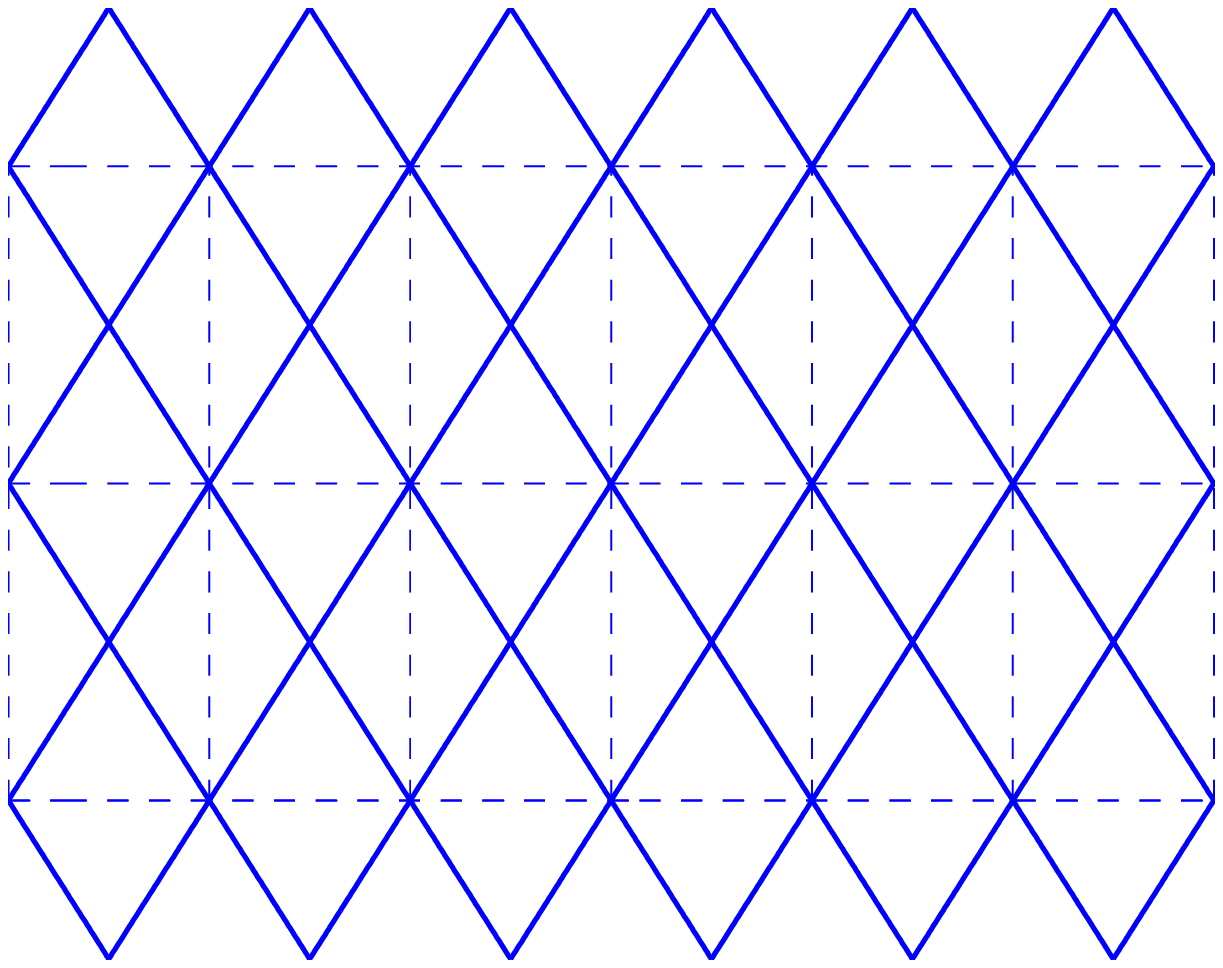}%
\caption{The grids on the polyhedron $\mathbb K_{6}(r,p/(p+1))$,  for $p=2$ (left) and $p=4$ (right). The dashed lines represent the edges of the prism.}\label{fig:gplanppar}
\end{figure}
For $p$ odd, we first rotate the prism $\mathbb P_{n}^{-}$ with $\pi/n$, around the axis $Oz$ and on each face of the prism we first draw $(p-1)/2$ horizontal lines, such that the resulting rectangles have the same area, except the ``lowest'' ones, which have only half of the area. Then we draw the rhombic cells as for the case $p$ even, except the ``lowest'' part of the prism (see Figure \ref{fig:gplanpimpar}).
\begin{figure}
\centering
\includegraphics[width=0.5\columnwidth]{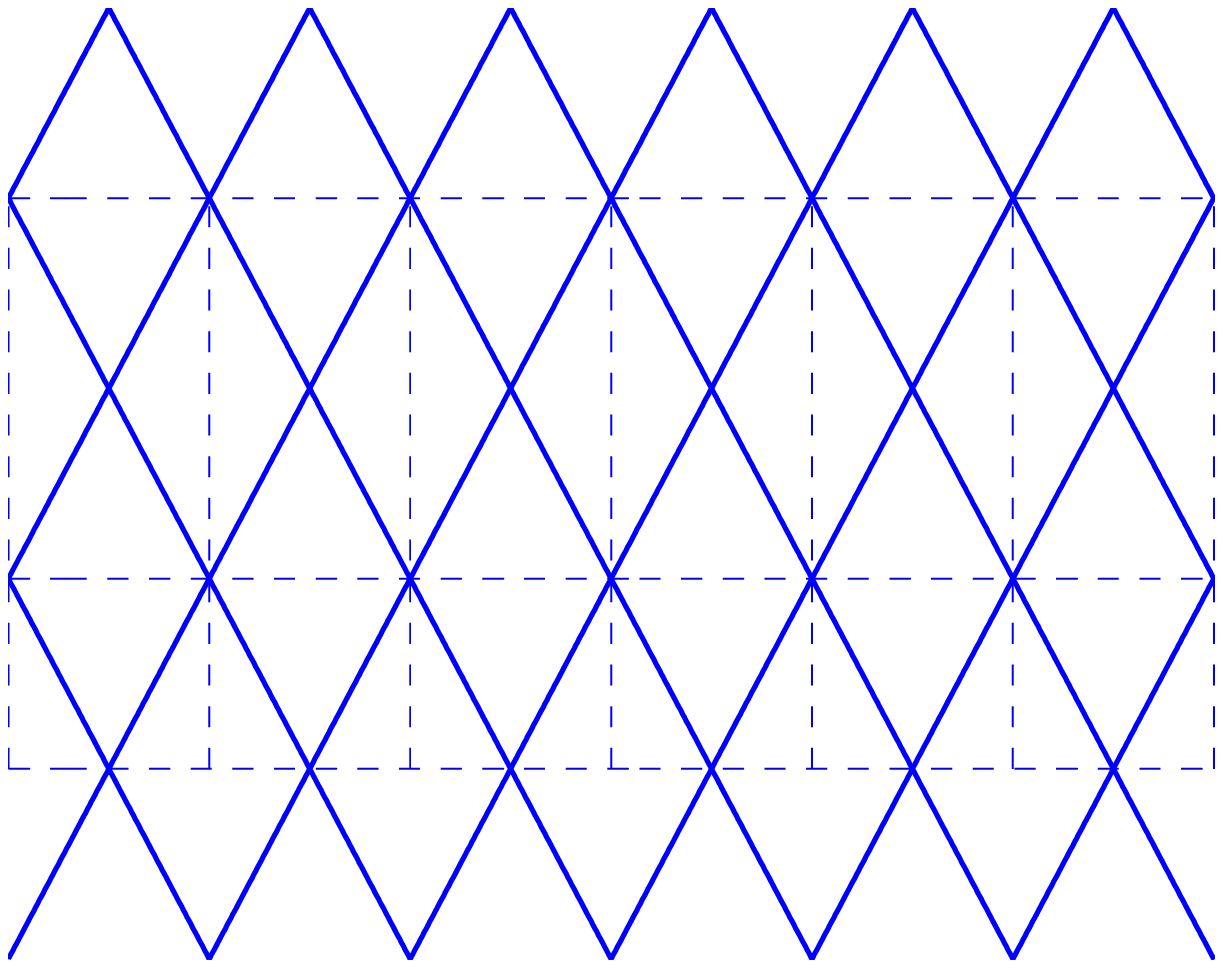}%
\includegraphics[width=0.5\columnwidth]{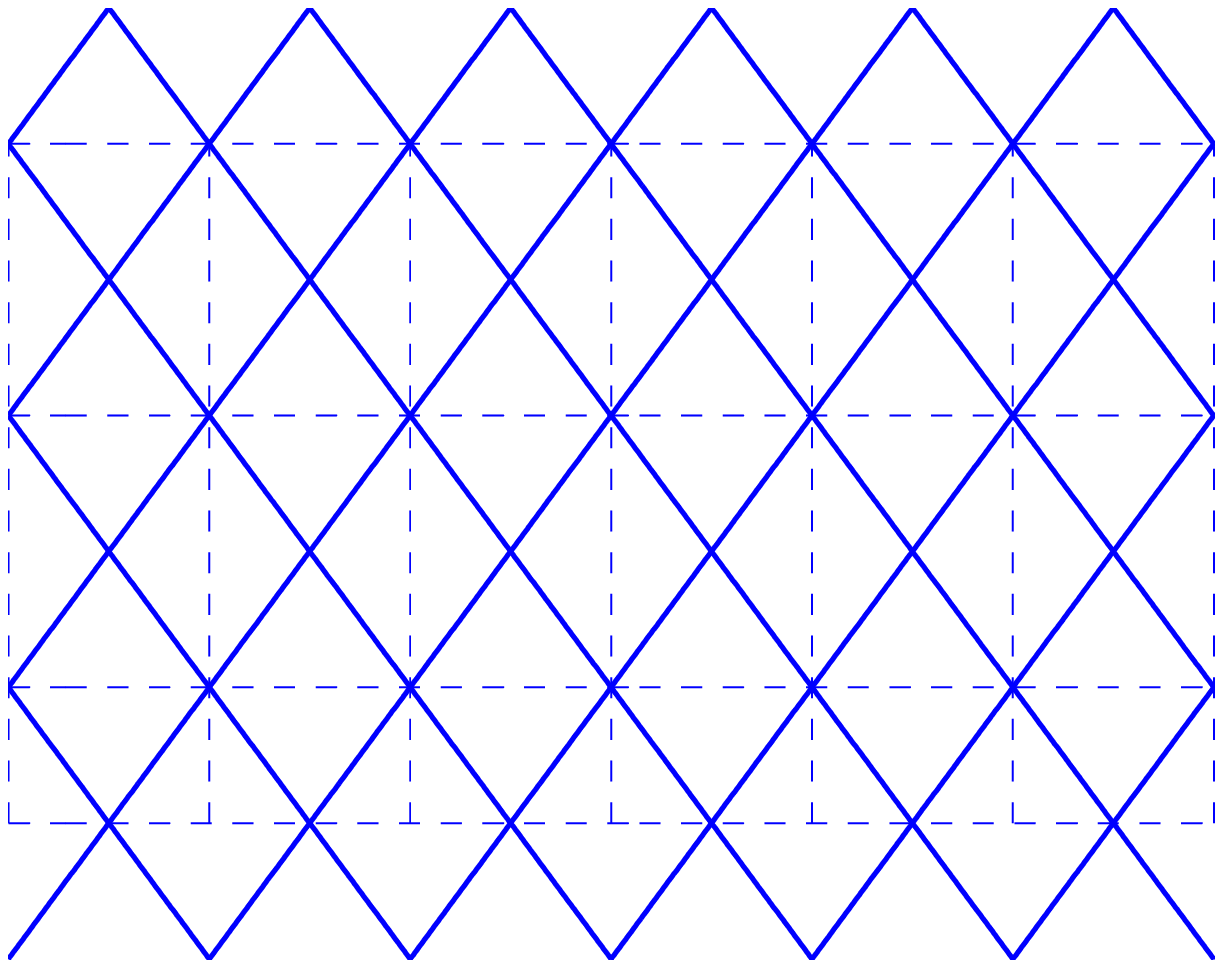}%
\caption{The grids on the polyhedron $\mathbb K_{6}(r,p/(p+1))$,  for $p=3$ (left) and $p=5$ (right). The dashed lines represent the edges of the prism.}\label{fig:gplanpimpar}
\end{figure}
The equalities in \eqref{deteps} also hold, so in this case we must have again $\varepsilon=p/(p+1)$.

Further, for $k\in \mathbb N$, we perform a subdivision of each rhombic cell of the grid on $\mathbb K_{n}(1,p/(p+1))$ into $k^{2}$ cells with the same area, by drawing $k-1$ equidistant parallel lines to each of the edges (see Figure 3).
\begin{figure}
\centering
\includegraphics[width=0.5\columnwidth]{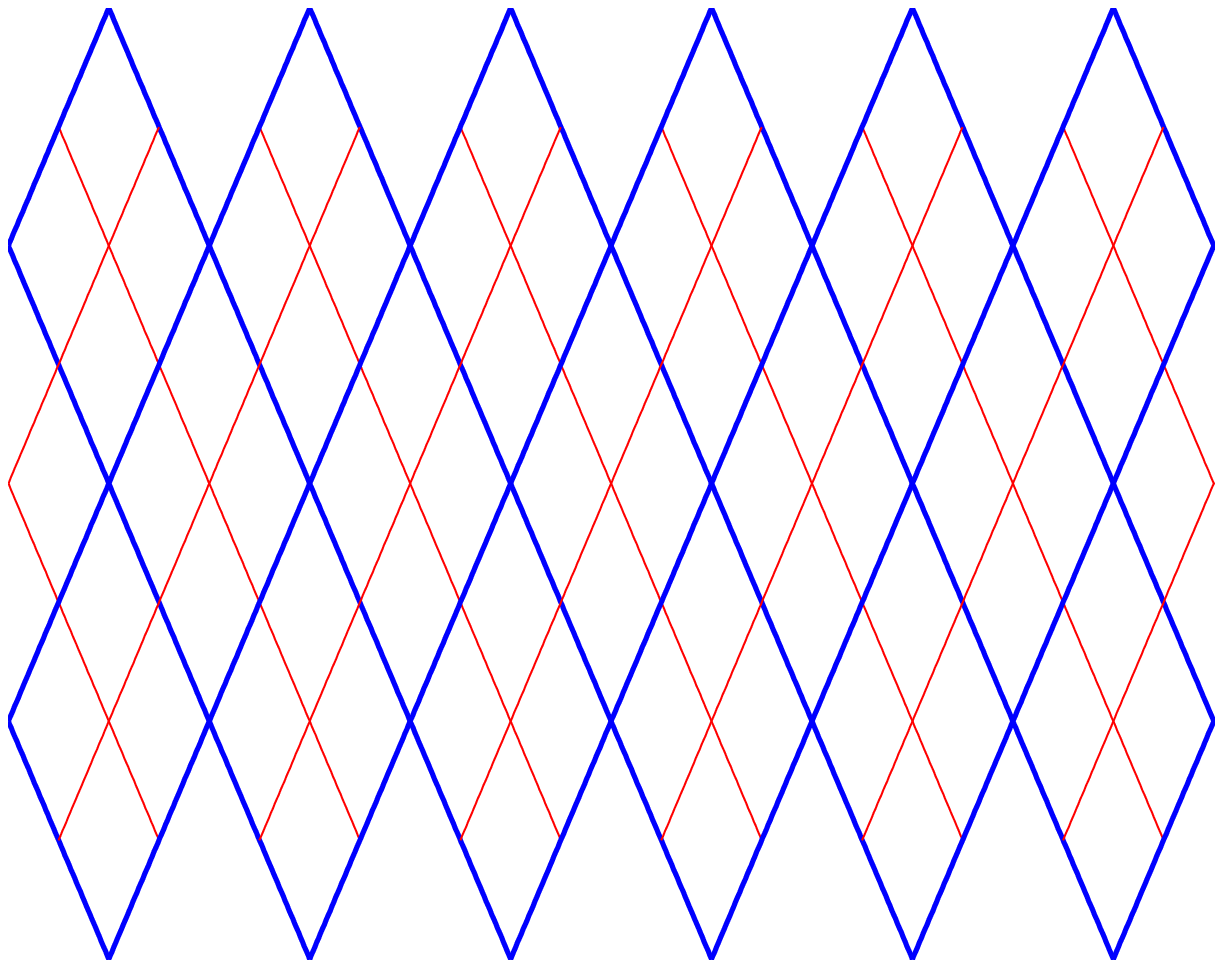}%
\includegraphics[width=0.5\columnwidth]{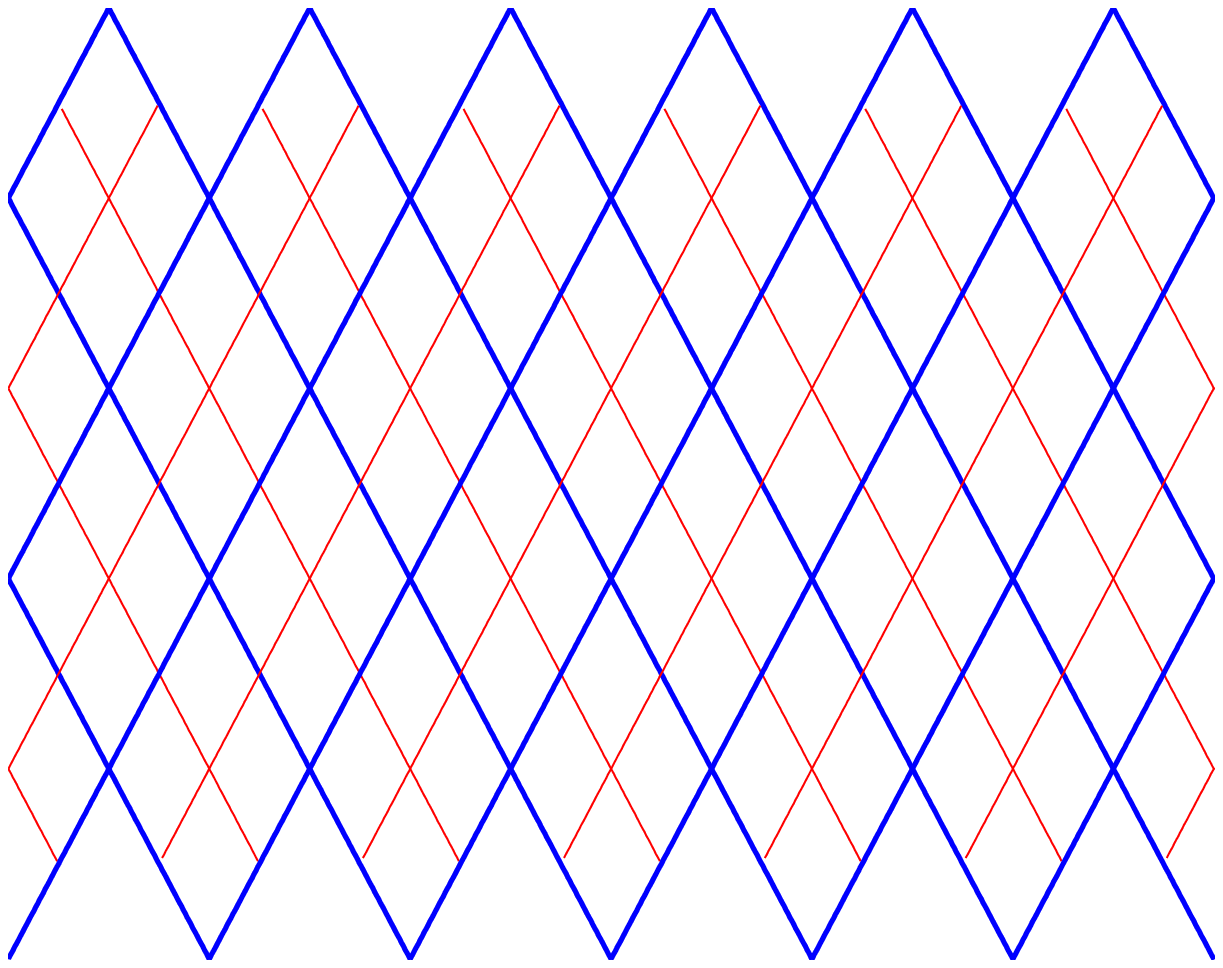}%
\caption{The refined grids on the polyhedron $\mathbb K_{6}(r,p/(p+1))$,  obtained after a subdivision with $k=2$, for $p=2$ (left) and $p=3$ (right).}\label{fig:gplanpimpar}
\end{figure}

Next we write the equations of the lines of our grids and we show that they coincide with the ones deduced in \cite{heal} for the HEALPix grids.

\subsubsection{The grid on the pyramid}
We parametrize the triangle with vertices
$$(0,0),\ (R_{n}\cos\alpha_i, R_{n}\sin\alpha_i),\ (R_{n}\cos\alpha_{i+1}, R_{n}\sin\alpha_{i+1})
$$ as
\begin{eqnarray*}
 X&=&u\cdot R_{n}\cos\alpha_i+v\cdot R_{n}\cos\alpha_{i+1},\\
Y&=&u\cdot R_{n}\sin\alpha_i+v\cdot R_{n}\sin\alpha_{i+1},
\end{eqnarray*}
with $u\in [0,1], \ v\in [0,1-u]$. For the parametrization of the face of the pyramid contained in $I_{i}^{+}$, we need to express $Z$ from the equation of the plane containing the points
$
(0,0, b_{n}+\varepsilon),\ (R_{n}\cos\alpha_i, R_{n}\sin\alpha_i,\varepsilon),\ (R_{n}\cos\alpha_{i+1}, R_{n}\sin\alpha_{i+1},\varepsilon).
$
Thus, we obtain
\begin{eqnarray*}
 Z&=&\varepsilon+b_{n}(1-u-v).
\end{eqnarray*}
%To obtain the grid of the cells of the pyramid: 1) we fix $a\in[0,1]$ and let $b\in[0,1-a]$ vary; 2) we fix $b\in[0,1]$ and let $a\in[0,1-b]$ vary.\par
The spherical images have the equations
\begin{eqnarray*}
x&=&\sqrt{1-z^2}\; \cos\left(\frac{2\pi}{n}\cdot\frac{v}{u+v}+\alpha_i \right),\\
y&=&\sqrt{1-z^2}\; \sin\left(\frac{2\pi}{n}\cdot\frac{v}{u+v}+\alpha_i \right),\\
 z&=&1-(1-\varepsilon)(u+v)^2.
\end{eqnarray*}
If we map the curves $u=\mathcal C_{1}, \ v\in [0,1-u]$ and $v=\mathcal C_{2},\  u\in [0,1-v]$ (with $\mathcal C_{1,2}$ constants chosen properly), we obtain the spherical curves in the HEALPix grid.
Indeed, if we take for example $n=4$ and $p=2$ (i.e. $\varepsilon={2}/{3}$), the above equations give
\begin{eqnarray}
z&=&1-\frac{1}{3}(u+v)^2,\label{hp1}\\
\theta&=& \frac{\pi}{2}\cdot \frac{v}{u+v}+\frac{\pi i}{2},\label{hp2}
\end{eqnarray}
where $\theta=\arctan (y/x)$. Let us consider $\theta_{t}=\theta \mod \frac{\pi}{2}=\frac{\pi}{2}\cdot \frac{v}{u+v}$. From \eqref{hp1} and \eqref{hp2} we obtain
\[z=1- \frac{1}{3}u^2\cdot \left(\frac{\pi}{2\theta_{t}-\pi}\right)^2, \ \mbox {and also }
z=1- \frac{1}{3}v^2\cdot \left(\frac{\pi}{2\theta_{t}}\right)^2.\]
%Thus, our grid consists in the following curves:
Now, if we consider $u=\frac \ell  {k}=:u_{\ell}$, for each $ \ell\in \{0,1,\ldots,k\}$ we obtain the family of curves
%\textcolor{magenta}{Inca nu mi-e clar care valori pt u si v dau care curbe, probabil dupa ce desenez cu matlab o sa ma dumiresc. Dar sper ca in continuare nu am scris prostii, de aceea te rog verifica.}
\begin{eqnarray*}
z&=&1- \frac{1}{3}u_{\ell}^2\cdot \left(\frac{\pi}{2\theta_{t}-\pi},\right)^2,\end{eqnarray*}
 and this is exactly the equation (20) from \cite{heal}. Note that
 $N_{\mathrm{side}}$ and $k$ in \cite{heal} are in our notations $k$ and $\ell$, respectively.
 %$N_{\mathrm{side}}=k$ and $k=\ell$.

The other class of curves of our grid are
\begin{eqnarray*}
v&=&\frac \ell k=:v_{\ell}\\
z&=&1- \frac{1}{3}v_{\ell}^2\cdot \left(\frac{\pi}{2\theta_{t}}\right)^2,
\end{eqnarray*}
which is exacly the equation (19) from \cite{heal}, with $v_{\ell}=k/N_{\mathrm{side}}$.
%------------------
\subsubsection{The grid on the prism}
The point $U \in \mathbb P_n^{+} \cap \mathbb B_{n} \cap   I_{i}^{+}  $  of coordinates
\begin{flalign*}
x_{U}&=\left(1-\frac{\ell}{k}\right)R_{n}\cos\frac{2\pi i}{n}+\frac{\ell}{k} R_{n}\cos\frac{2(i+1)\pi}{n},\\
y_{U}&=\left(1-\frac{\ell}{k}\right)R_{n}\sin\frac{2\pi i}{n}+\frac{\ell}{k} R_{n}\sin\frac{2(i+1)\pi}{n},\\
z_{U}&=\varepsilon,
\end{flalign*}
and $Q\in \mathbb P_n^{-} \cap \mathbb B_{n} \cap   I_{i}^{-}$ of coordinates
\begin{flalign*}
x_{Q}&=\left(1-\frac{\ell}{k}\mp\frac{p}{2}\right)R_{n}\cos\frac{2\pi i}{n}+\left(\frac{\ell}{k}\pm\frac{p}{2}\right) R_{n}\cos\frac{2(i+1)\pi}{n},\\
y_{Q}&=\left(1-\frac{\ell}{k}\mp\frac{p}{2}\right)R_{n}\sin\frac{2\pi i}{n}+\left(\frac{\ell}{k}\pm\frac{p}{2}\right) R_{n}\sin\frac{2(i+1)\pi}{n},\\
z_{Q}&=-\varepsilon.
\end{flalign*}
belong to the same line of the grid of the prism.
%\begin{figure}%
%\centering
%\includegraphics[width=0.5\columnwidth]{uq.eps}%
%\caption{$k=6$, $p=2$, $\ell=2$}%
%\label{uq}%
%\end{figure}
A point $(X,Y,Z)$ of the grid of the prism satisfies the equation of the line passing through $U$ and $Q$
\[\frac{X-x_{U}}{x_{Q}-x_{U}}=\frac{Y-y_{U}}{y_{Q}-y_{U}}=\frac{Z-z_{U}}{z_{Q}-z_{U}}=t.\]
When we calculate its image on the sphere using formulas (\ref{xeq}), (\ref{yeq}) and (\ref{zeq}), we obtain
\begin{flalign*}
z&=\varepsilon-2\varepsilon t,\\
x&=\sqrt{1-z^2}\; \cos\left(\frac{2\pi}{n}\cdot \left(i+\frac{\ell}{k}\pm \frac{pt}{2}\right)\right),\\
y&=\sqrt{1-z^2}\; \sin\left(\frac{2\pi}{n}\cdot \left(i+\frac{\ell}{k}\pm \frac{pt}{2}\right)\right).
\end{flalign*}
In the particular case $p=2, \ n=4,\ \varepsilon=2/3$ we obtain
\begin{flalign*}
z&=\frac{2}{3}-\frac{4t}{3},\\
\theta&=\frac{\pi}{2}\left(i+\frac{\ell}{k}\pm t\right).
\end{flalign*}
If we take again $\theta_{t}=\theta \mod \frac{\pi}{2}=\frac{\pi}{2}\cdot \left(\frac{\ell}{k}\pm t\right)$, we further obtain
\[z=\frac{2}{3}\pm \frac{4}{3}\left(\frac{2\theta_{t}}{\pi}-\frac{\ell}{k} \right), \]
which agree with formulas (22), (A2), (A3) in \cite{heal}.

\section{Volume preserving map from the interior of a polyhedron $\mathbb K_n(r',\varepsilon)$ to the ball of radius $r$.}\label{sec:volpres}

We will use the area preserving map constructed in the previous section for the construction of a volume preserving map from the interior of a polyhedron $\mathbb K_n$ onto the ball of radius $r$.

We fix $\varepsilon$, and for $\rho>0$ we define
$$\overline{\mathbb K}_n(\rho, \varepsilon)=\mbox{int}(\mathbb K_n(\rho,\varepsilon)) \cup \mathbb K_n(\rho,\varepsilon) \quad \mbox {and}$$
$$\overline{\mathbb S^2}(\rho)=\{(x,y,z)\in \mathbb R^3,\ x^2+y^2+z^2\leq \rho^2  \}.$$
Their volumes are
\begin{equation}\nonumber
 vol(\overline{\mathbb K}_n(\rho,\varepsilon))=\rho^3 \gamma, \qquad
vol(\overline{\mathbb S^2}(\rho))=\rho^3\beta,
\end{equation}
with
$$ \gamma=2 \big (\varepsilon + \frac {c (\varepsilon)}3 \big )\frac{\pi^2}n\cot\frac{\pi}n, \quad
c{(\varepsilon)}=\sqrt{4(1-\varepsilon)^2 -\frac {\pi^2}{n^2}\cot^2 \frac \pi n}, \quad\beta=\frac{4\pi}3.
$$
We also need to define the domains
\begin{eqnarray*}
\overline{\mathbb S^{+}}(\rho)&=&\overline{\mathbb S^{2}}(\rho)\cap \{ (x,y,z), \ z\geq \varepsilon \rho\},\\
\overline{\mathbb E}(\rho)&=&\overline{\mathbb S^{2}}(\rho)\cap \{ (x,y,z), \ | z|\leq \varepsilon \rho\}.
\end{eqnarray*}
\subsection{Construction of the volume preserving map  $\mathcal V_{n}:\overline{\mathbb S^2}(r) \to \overline{\mathbb K}_n(r\sqrt[3]{\beta/\gamma},{\varepsilon})$ }
Next we fix $r>0, \, n\in \mathbb N$, $n\geq 3$, and we will investigate the possibility of construction of a map $\mathcal V_{n}:\overline{\mathbb S^2}(r) \to \overline{\mathbb K}_n(r\sqrt[3]{\beta/\gamma},{\varepsilon})$ which is volume preserving, i.e. $$vol(D)=vol(\mathcal V_{n}(D)), \quad \mbox{for all }D\subseteq \overline{\mathbb S^2}(r).$$
We propose the following construction, which has two steps.

\textit{Step 1}. We denote $\xi=\sqrt[3]{\beta/\gamma}$ and define the map $\mathcal L:\overline{\mathbb S^2}(r) \to \overline{\mathbb S^2}(\xi r)$,
$$
\mathcal L (x,y,z)=\xi\cdot(x,y,z), \mbox{ for all }(x,y,z)\in \overline{\mathbb S^2}(r).
$$

\textit{Step 2}. Let $(\widetilde{x},\widetilde{y},\widetilde{z})\in \overline{\mathbb S^2}(r\xi)$. Then $(\widetilde{x},\widetilde{y},\widetilde{z})\in \mathbb S^2(\widetilde{\rho})$, with $\widetilde{\rho}=\sqrt{\widetilde{x}^2+\widetilde{y}^2+\widetilde{z}^2}=\xi \rho$. Now we consider the area preserving map $\mathcal T_{n,\widetilde \rho}:\mathbb S^2(\widetilde{\rho}) \to \mathbb K_n(\widetilde{\rho},\varepsilon)$ defined in Section \ref{sec:conTn} (in formulas \eqref{fiplusX}-\eqref{fiplusZ}, \eqref{fX-}-\eqref{fZ-} and \eqref{fe1}-\eqref{fe3} we have to take $\widetilde{\rho}$ instead of $r$).

Finally, we define the map $\mathcal V_{n}:\overline{\mathbb S^2}(r) \to \overline{\mathbb K}_n(r \sqrt[3]{\beta/\gamma},\varepsilon)$ by $$\mathcal V_{n}( x, y, z)=(\mathcal T_{n,\widetilde \rho} \circ \mathcal L)( x, y, z)=(\mathcal T_{n,\xi( x^{2}+y^{2}+  z^{2})^{1/2}} \circ \mathcal L) (x,y,  z).$$

Next we investigate wether the Jacobian of $\mathcal V_{n}$ ca be $\pm1$.

\textit{Case 1}. For $(x,y,z)\in \overline{\mathbb S^+}(r) \cap I_0^+$ the formulas for $(X,Y,Z)=\mathcal V_{n}(x,y,z)$ are
\begin{eqnarray}
X&=&\xi \sqrt {\frac{\rho(\rho-z)}{1-\varepsilon }}\left(\frac{\pi}{n\sin \frac \pi n}- \sin\frac \pi n \arctan \frac yx \right),\label{volX}\\
Y&=&\xi \sqrt {\frac{\rho(\rho-z)}{1-\varepsilon }}\,\cos \frac {\pi}{n}\arctan \frac yx,\label{volY}\\
Z&=&\xi\rho \left( \varepsilon +c (\varepsilon) \left( 1-\sqrt \frac{1-z/\rho}{1-\varepsilon } \right)\right),\label{volZ}
\end{eqnarray}
where $\rho=\sqrt{x^2+y^2+z^2}.$ If we evaluate the Jacobian of $\mathcal V_{n}$ restricted to $I_0^+$ we find
$$
J(\mathcal V_{n})=\left |\begin{array}{ccc}
                       \frac{\partial X}{\partial x} & \frac{\partial X}{\partial y} & \frac{\partial X}{\partial z} \\
                       \frac{\partial Y}{\partial x} & \frac{\partial Y}{\partial y} & \frac{\partial Y}{\partial z} \\
                       \frac{\partial Z}{\partial x} & \frac{\partial Z}{\partial y} & \frac{\partial Z}{\partial z}
                     \end{array}
 \right | =\frac{c(\varepsilon)+\varepsilon}{(1-\varepsilon)(c(\varepsilon)+3\varepsilon)}>0.
$$
If we impose the volume preserving condition $J(\mathcal V_{n})=1$ we obtain the solution $\varepsilon=0$ for all $n\geq 3$. Another  case when $J(\mathcal V_{n})=1$ is when $c(\varepsilon)=2-3\varepsilon$, which reduces to the equation
\begin{equation}
2-3\varepsilon =\sqrt{ 4(1-\varepsilon)^{2}-\frac {\pi^2}{n^2} \cot^2 \frac {\pi}{n}.}
\end{equation}
This equation gives solutions only in the case when $n=3,4,5$, and the solutions are
\begin{description}
\item for $n=3$,
$$\varepsilon=\frac 1{45}\left( 18-\sqrt {3(108-5\pi^{2})}\right), \mbox{ i.e. }\varepsilon=0.105226...;$$
\item for $n=4$,
$$
\varepsilon_{1,2}=\frac 1{20}\left( 8\pm \sqrt {64-5\pi^{2})}\right), \mbox{ i.e. }\varepsilon_1=0.20861..., \ \varepsilon_{2}=0.59139...;
$$
\item for $n=5$,
$$
\varepsilon_{1,2}=\frac 1{25}\left( 10\pm \sqrt {100-5\pi^{2}-2\sqrt 5 \pi^{2})}\right), \mbox{ i.e. }\varepsilon_1=0.297912..., \ \varepsilon_{2}=0.502088... .
$$
\end{description}
\textit{Case 2.} For $(x,y,z)\in \overline{\mathbb E}(r) \cap I_0^+$ the formulas for $(X,Y,Z)=\mathcal V_{n}(x,y,z)$ are
\begin{eqnarray*}
X&=&\xi \rho \left(\frac{\pi}{n\sin \frac \pi n}- \sin\frac \pi n \arctan \frac yx \right),\\
Y&=&\xi \rho \cos \frac \pi n  \arctan \frac yx ,\\
Z&=&\xi z.
\end{eqnarray*}
In this case we calculate that
$$J(\mathcal V_{n})=\frac 2 {3 \varepsilon +c(\varepsilon)},$$
and its value is 1 when $c(\varepsilon)=2-3\varepsilon.$
\medskip

Of course, it is immediate that the Jacobian is 1 for the whole ball $\overline{\mathbb S^2}(r).$
\begin{remark}
The method that we have used above, which makes use of the area preserving map, can be applied only in the cases when all the polyhedrons $\mathbb K_{n}(\widetilde \rho,\varepsilon)$, with $\widetilde \rho \in (0,r\xi]$ admit inside a sphere which is tangent to all the faces. Indeed, if we split $\mathbb K_{n}(\widetilde \rho,\varepsilon)$ with $\varepsilon \neq 0$ into small pyramids with apex at the origin $O$ and bases the faces of $\mathbb K_{n}(\widetilde \rho,\varepsilon)$, there exists such a sphere if the distance from $O$ to a face of $\mathbb P_{n}^{+}$ equals to the distance from $O$ to a face of $\mathbb B_{n}$, i.e.
\begin{equation}\label{circum}
\frac{(b_{n}+\varepsilon \widetilde \rho) \pi}{2(1-\varepsilon)n}\cot \frac \pi n=\frac {\pi \widetilde \rho }n \cot \frac \pi n=:d_{n}.
\end{equation}
In this case, the sum of the volumes of the small pyramids is
$$vol(\mathbb K_{n}(\widetilde \rho,\varepsilon))=\frac{d_{n}\mathcal A(\mathbb K_{n}(\widetilde \rho,\varepsilon))}3=\frac{d_{n}\mathcal A (\mathbb S^{2}(\widetilde \rho)) }3=\frac{d_{n} 4\pi \widetilde \rho^{2}}3.$$
This equals the volume of the sphere $\mathbb S^{2}(\widetilde \rho/\xi)$ if $d_{n}=\widetilde \rho /\xi^{3}$, which reduces to $c(\varepsilon)=2-3\varepsilon.$

On the other hand, condition \eqref{circum} further gives $b_{n}=\widetilde \rho (2-3 \varepsilon)$, i.e. again $c(\varepsilon)=2-3 \varepsilon.$ This condition is exactly the condition obtained by imposing $J(\mathcal V_{n})=1.$

We remind also the construction in \cite{RMG}, where a similar method worked for the cube, which is also a polyhedron which can be circumscribed to a sphere.
\end{remark}
\subsection{The inverse $\mathcal V_{n}^{-1}:  \overline{\mathbb K}_n(r \sqrt[3]{\beta/\gamma},\varepsilon) \to \overline{\mathbb S^2}(r)$}

\subsubsection{The case $\varepsilon=0$}
Let $(X,Y,Z)\in \overline{\mathbb K}_n(r \sqrt[3]{\beta/\gamma},0)  $.  First we have to find the domain $I_{i}^{s}$ which contains $(X,Y,Z)$.  Suppose, for simplicity, that
$(X,Y,Z)\in I_0^+$.
We have to find $\overline \rho\in (0,r \xi]$ such that $(X,Y,Z)\in \mathbb K_n(\overline \rho,0)$. The equation of the plane containing the face of $\mathbb P_{n}^{+}\cap I_{0}^{+}$ is
\begin{equation}\label{plan}
X+Y \tan\frac \pi n+Z \frac \pi {n \sin \frac \pi n \sqrt{4-\frac {\pi^2}{n^2}\cot^{2}\frac \pi n}}=\frac {\pi \overline \rho}{n\sin \frac \pi n}.
\end{equation}
Therefore, when $(X,Y,Z)$ is given, then $\overline \rho$ can be calculated from formula \eqref{plan}.  Further, $\mathcal V_{n}^{-1}(X,Y,Z)$ is defined as
\begin{equation}\label{inversaV}
\mathcal V_{n}^{-1}(X,Y,Z)=(\mathcal L^{-1} \circ \mathcal T^{-1}_{n,\overline \rho})(X,Y,Z),
\end{equation}
which can be immediately be calculated using the formulas \eqref{fz},\eqref{fx},\eqref{fy} for $\alpha_{i}=0, \varepsilon=0$, and $r=\overline \rho.$ The formulas for the general case $(X,Y,Z)\in I_{i}^{\pm}$ cand be deduced similarly using the same ideas as in Section \ref{sec:conTn}.

\subsubsection{The case $\varepsilon \neq 0$}
For simplicity, suppose again that $(X,Y,Z)\in I_{0}^{+}\cup I_{0}^{-}.$ In this case we have to find not only the radius $\overline \rho$ of the polyhedron $\mathbb K_{n}(\overline \rho,\varepsilon)$ which contains the point $(X,Y,Z)$, but we must also find whether $(X,Y,Z)$ belongs to the pyramid $\mathbb P_{n}^{+}(\overline \rho,\varepsilon)$ or to the prism $\mathbb B_{n}(\overline \rho,\varepsilon)$.
First we observe that the pyramids $\mathbb P_{n}^{+}(\rho,\varepsilon)$, $\rho>0$ are situated above  the conical surfaces generated by lines that make a constant angle $\beta=\arccos \varepsilon $ with the axis $OZ$. This means that $(X,Y,Z)\in \mathbb P_{n}^{+}(\overline \rho,\varepsilon)$ if it satisfies the equations
$$
X^{2}+Y^{2}\leq \left( \frac 1{\varepsilon^{2}}-1\right)Z^{2},\qquad Z\geq 0.
$$
and $(X,Y,Z)\in \mathbb P_{n}^{-}(\overline \rho,\varepsilon)$ if it satisfies the first inequality and $Z\leq 0.$ If
$$
X^{2}+Y^{2}\geq \left( \frac 1{\varepsilon^{2}}-1\right)Z^{2},
$$
then $(X,Y,Z)\in \mathbb B_{n}(\overline \rho,\varepsilon). $

Further we have to calculate $\overline \rho.$ In the case when $(X,Y,Z)\in \mathbb P_{n}^{\pm} (\overline \rho,\varepsilon)$, the determination of $\overline \rho$ was already described for the case $\varepsilon =0$, but in this case $\overline \rho$ is calculated from the formula
$$
X+Y \tan\frac \pi n+Z \frac \pi {c(\varepsilon) n \sin \frac \pi n}=\overline \rho \left(
\frac \varepsilon {c(\varepsilon)}+1
\right)\frac {\pi }{n\sin \frac \pi n}.
$$

So it remains to consider the case  $(X,Y,Z)\in \mathbb B_{n}(\overline \rho,\varepsilon)$. The equation of the vertical  plane which contains $(X,Y,Z)$ can be written as
\begin{equation}\label{detroprisma}
X+Y\tan \frac \pi n=\frac{\pi \overline \rho}{n \sin \frac \pi n},
\end{equation}
and this formula allows us to calculate the value of $\overline \rho$.

Finally, after the determination of $\overline \rho$, the inverse can be calculated with formula \eqref{inversaV}.

\subsection{Transporting uniform and refinable grids from the polyhedrons $\overline{\mathbb K}_n(r',\varepsilon)$ onto the ball}
The most important application of the volume preserving map $\mathcal V_n$ is that it allows the construction of uniform and refinable grids on the ball starting from similar grids of the polyhedron. In a first step we construct $2n$ tetrahedrons (triangular prisms), each having the vertex $O$ and base a face of the pyramids $\mathbb P_n^\pm$. For the prism $\mathbb B_n$, we triangularize each face of $\mathbb B_n$ and then construct again $2n$ tetrahedrons with vertex $O$ and base a triangle. In the second step, we split every of the $4n$ tetrahedrons into four smaller tetrahedrons having the same volume, this simple procedure being described in \cite{HR2}. The refinement can be repeated, therefore we end up with a uniform and refinable grid of $\overline{\mathbb K}_n(r',\varepsilon)$, whose image on the ball $\overline {S^2}(r)$ will be a uniform and refinable grid.

%===========================


\begin{thebibliography}{99}
\bibitem{A72} R. Alexander, \textit{On the sum of distances between $N$ points on the sphere,} Acta. Math. Hungar., 23 (1972), 443--448.
\bibitem{heal} K. M. G\'orski, B. D. Wandelt, E. Hivon, A. J. Banday, B. D. Wandelt, F. K. Hansen, Reinecke, M. M. Bartelmann, \textit{HEALPix: A framework for high-resolution discretization and fast analysis of data distributed on the sphere}, Astrophys. J., 622 (2005), 759--771.
\bibitem{Gra} E. W. Grafared and F. W. Krumm, \textit{Map Projections, Cartographic Information Systems}, Springer-Verlag, Berlin, 2006.
\bibitem{HR} A. Holho\c s, D. Ro\c sca, \textit{An octahedral equal area partition of the sphere and near optimal configurations of points}, Comput. Math. Appl., vol. 67, 5 (2014), 1092-1107.
\bibitem{HR2}  A. Holho\c s, D. Ro\c sca, \textit{Equal-volume subdivisions of regular convex polyhedrons and balls}, submitted.
\bibitem{Leo} P. Leopardi, A partition of the unit sphere into regions of equal area and small diameter, Electron. Trans. Numer. Anal. 25 (2006), 309--327.
\bibitem{RMC}D. Ro\c sca, \textit{Locally supported rational spline wavelets on the sphere}, Math. Comp. 74, 252 (2005), 1803-1829.
\bibitem{RoRM} D. Ro\c sca, \textit{On a norm equivalence in $L^{2}(\mathbb S^{2})$}, Result. Math. 53, 3-4 (2009), 399-405.
\bibitem{Ro} D. Ro\c{s}ca, \textit{New uniform grids on the sphere, Astron. Astrophys.}, 520 (2010), A63.
 \bibitem{RP} D. Ro\c{s}ca, G. Plonka, \textit{Uniform spherical grids via equal area projection from the cube to the sphere}, J. Comput. Appl. Math., 236, 6 (2011), 1033-1041.
\bibitem{RP2} D. Ro\c sca, G. Plonka, \textit{An area preserving projection from the regular octahedron to the sphere}, Result. Math., 63, 2 (2012), 429-444.
\bibitem{RMG}D. Ro\c sca, A. Morawiec and M. De Graef, \textit{A new method of constructing a grid in the space of 3D rotations and its applications to texture analysis}, Modelling Simul. Mater. Sci. Eng. 22 (2014) 075013 (17pp).
\bibitem{Sny} J. P. Snyder, \textit{An equal-area map projection for polyhedral globes,} Cartographica: The International Journal for Geographic Information and Geovisualization, 29, 1 (1992), 10--21.
\bibitem{SKS} L. Song, A. J. Kimerling, and K. Sahr, \textit{Developing an equal area global grid by small circle subdivision, in Discrete Global Grids}, M. Goodchild and A. J. Kimerling, eds., National Center for Geographic Information \& Analysis, Santa Barbara, CA, USA, 2002.
\bibitem{Syd}  J. P. Snyder, \textit{Flattening the Earth}, University of Chicago Press, 1990.
\bibitem{T96} M. Tegmark, \textit{An icosahedron-based method for pixelizing the celestial sphere}, ApJ. Letters, 470 (1996), L81.
\bibitem{T06} N.A. Teanby, \textit{An icosahedron-based method for even binning of globally distributed remote sensing data}, Comput. \& Geosci. 32, 9 (2006), 1442--1450.
%\bibitem{Z95} Y. M. Zhou, \textit{Arrangements of points on the sphere}, PhD thesis, Mathematics, Tampa, FL 1995.
\end{thebibliography}
\end{document}